\RequirePackage{fix-cm}
\documentclass[smallextended]{svjour3}       
\smartqed  
\usepackage{graphicx}
\usepackage{amssymb}
\usepackage{amsmath}
\usepackage{relsize}

\newcommand{\NN}{{\mathbb N}}

\journalname{Periodica Mathematica Hungarica}
\newtheorem{Theorem}{Theorem}
\newtheorem{Lemma}{Lemma}
\newtheorem{Corollary}{Corollary}
\newtheorem{Remark}{Remark}

\def \NN{\mathbb N}

\def \Iq{\mathcal I_{<q}}
\def \Icq{\mathcal I_c^{(q)}}
\def \Ieq{\mathcal I_{\leq q}}
\def \I0{\mathcal I_0}
\def \Ilim#1{\mathop{#1\text{\rm--lim}}\limits}
\DeclareMathSymbol{\shortminus}{\mathbin}{AMSa}{"39}

\usepackage{mathpazo} 

\newcommand\rsetminus{\mathbin{\mathpalette\rsetminusaux\relax}}
\newcommand\rsetminusaux[2]{\mspace{-4mu}
	\raisebox{\rsmraise{#1}\depth}{\rotatebox[origin=c]{-20}{$#1\smallsetminus$}}
	\mspace{-4mu}
}
\newcommand\rsmraise[1]{%
	\ifx#1\displaystyle .8\else
	\ifx#1\textstyle .8\else
	\ifx#1\scriptstyle .6\else
	.45%
	\fi
	\fi
	\fi}

\begin{document}
	
	\title{On $\Iq\shortminus$ and $\Ieq\shortminus$convergence of arithmetic functions
}


\author{J\'anos~T.~T\'oth       \and Ferdin\'and Filip \and J\'ozsef Bukor 
	\and  L\'aszl\'o Zsilinszky
}


\institute{
		J\'anos~T.~T\'oth \at
		\email{tothj@ujs.sk}
	\and
	Ferdin\'and Filip \at
	\email{filipf@ujs.sk}           
	\and
	J\'ozsef Bukor  \at
	\email{bukorj@ujs.sk}
		\\ \\
	Department of Mathematics and Informatics\\
	J. Selye University\\ 945 01 Kom\'arno,
	Slovakia
	\and
	L\'aszl\'o Zsilinszky \at \email{laszlo@uncp.edu}
	\\ \\
		Department of Mathematics and Computer Science\\
		University of North Carolina at Pembroke\\
		Pembroke, NC 28304, U. S. A.\\	
}

\date{Received: date / Accepted: date}

\maketitle

\begin{abstract}
 Let $\NN$ be the set of positive integers, and denote by $$\lambda(A)=\inf\{t>0:\sum_{a\in A} a^{-t}<\infty\}$$ the convergence exponent of $A\subset\NN$. For $0<q\le 1$, $0\le q\le 1$, respectively,   the admissible ideals $\Iq$, $\Ieq$ of all subsets $A\subset \NN$ with $\lambda(A)<q$, $\lambda(A)\le q$, respectively, satisfy $\Iq\subsetneq\Icq\subsetneq \Ieq$, where $$\Icq=\{A\subset\NN: \sum_{a\in A}a^{-q}<\infty\}.$$ In this note we sharpen the results of
 Bal\'a\v z, Gogola and Visnyai from \cite{2}, and other papers, concerning characterizations of $\Icq$\nobreakdash-convergence of various arithmetic functions in terms of $q$. This is achieved by utilizing  $\Iq$\nobreakdash- and $\Ieq$\nobreakdash-convergence, for which new methods and criteria are developed.

\end{abstract}


\section{Introduction}
Denote by $\NN$ the set of positive integers, and let $\lambda$ be  the convergence exponent function on the power set $2^{\NN}$ of $\NN$, i.e. for $A\subset \NN$ put
\[
\lambda(A)=\inf\Big\{t>0:\sum_{a\in A}\frac{1}{a^t}<\infty\Big\}.
\]
If $q>\lambda(A)$ then $\sum_{a\in A}\frac{1}{a^q}<\infty$, and $\sum_{a\in A}\frac{1}{a^q}=\infty$ when $q<\lambda(A)$; if
$q=\lambda(A)$, the convergence of $\sum_{a\in A}\frac{1}{a^q}$ is inconclusive. It follows from [13, p.26, Exercises 113, 114] that the range of  $\lambda$ is the interval $[0,1]$, moreover for $A=\{a_1<a_2<\cdots<a_n<\dots\}\subset\NN$,
\[
\lambda(A)=\limsup_{n\to\infty}\frac{\log n}{\log a_n}.
\]
It is easy to see that $\lambda$ is monotonic, i.e. $\lambda(A)\leq\lambda(B)$ whenever $A\subset B\subset\NN$, furthermore, $\lambda(A\cup B)=\max\{\lambda(A),\lambda(B)\}$ for all $A,B\subset\NN$.
Define the following sets:
\begin{align*}
\Iq&=\{A\subset\NN:\lambda(A)<q\}, \  \text{if} \ 0<q\leq 1, \\
\Ieq&=\{A\subset\NN:\lambda(A)\leq q\}, \  \text{if} \ 0\leq q\le 1, \ \text{and} \\
\I0&=\{A\subset\NN:\lambda(A)=0\}.
\end{align*}
Clearly, $\mathcal I_{\leq 0}=\I0$, and  $\mathcal I_{\leq 1}=2^\NN$.
Since $\lambda(A)=0$ when $A\subset\NN$ is finite, then  $\mathcal I_f=\{A\subset\NN: \textrm{$A$ is finite}\}\subset\I0$, moreover, also considering the well-known set
\[
\Icq=\Big\{A\subset\NN: \sum_{a\in A}\frac{1}{a^q}<\infty\Big\}
\]
we get that whenever  $0<q<q'<1$,
\begin{equation}
\label{1}
\mathcal I_f\subset\I0\subset\Iq\subset\Icq\subset \Ieq\subset \mathcal I_{<q'}.
\end{equation}
In what follows, we will use the following definitions.

The set $\mathcal I\subseteq 2^\NN$ is a so-called  admissible ideal, provided  $\mathcal I$ is additive (i.e. $A,B\in\mathcal I$ implies $A\cup B\in\mathcal I$), hereditary (i.e. $A\in\mathcal I,\ B\subset A$ implies $B\in\mathcal I$), it contains the singletons, and $\NN\notin\mathcal I $.

Given an ideal $\mathcal I\subset 2^{\NN}$, we say that a sequence $x=(x_n)_{n=1}^\infty$ $\mathcal I$\nobreakdash-converges to a number $L$,	and  write $\Ilim{\mathcal I}  x_n=L$, if for each $\varepsilon >0$ the set \begin{equation}\label{Ae} A_\varepsilon =\{n:|x_n - L|\geq\varepsilon\}
\end{equation}
 belongs to the ideal $\mathcal I$. One can see \cite{7}, \cite{9} for a general treatment of $\mathcal I$\nobreakdash-convergence; a useful property is as follows:
\begin{Lemma}[\cite{9}]
	If $\mathcal I_1\subset\mathcal I_2$, then  $\Ilim{\mathcal I_1} x_n=L$ implies  $\Ilim{\mathcal I_2}  x_n=L$.
\end{Lemma}
We will study $\mathcal I$\nobreakdash-convergence in the case when $\mathcal I$ stands for $\Iq$, $\Icq$, $\Ieq$, respectively. We will establish necessary and sufficient conditions for a set $A\subset \NN$ to belong to $\Iq$, $\Ieq$, respectively; as well as for the set $A_\varepsilon =\{n:|x_n - L|\geq\varepsilon\}$ so that $\Ilim{\Iq} x_n=L$, resp. $\Ilim{\Ieq} x_n=L$ hold. Note that analogous criteria were not known for $\Icq$.

In this paper, we embed the ideals $\Iq$ and $\Ieq$ into the structure of ideals $\Icq$. We show that theses ideals are essentially distinct. Then we refine a known statement concerning the  $\Icq$-convergence of some arithmetic functions.
A new method is introduced and can be applied widely for consideration of $\Iq$ and $\Ieq$-convergence of sequences.

\section{On ideals enveloping the ideal $\Icq$}

\begin{Theorem}\label{beagyazas} Let $0<q<q'< 1$. Then 
\begin{equation}	
\label{sorba}
\I0\subsetneq\Iq\subsetneq\Icq\subsetneq \Ieq\subsetneq \mathcal I_{<q'}\subsetneq\mathcal I_c^{(q')}\subsetneq \mathcal I_{\leq q'}
\subsetneq\mathcal I_{<1}\subsetneq\mathcal I_c^{(1)}\subsetneq\mathcal I_{\leq 1}=2^\NN.
\end{equation}
\end{Theorem}

\begin{proof}  

The inclusions follow from the definitions of the sets.
We can show that the difference of successive sets in \eqref{sorba} is infinite, so equality does not hold in any of the inclusions, by considering the following four cases (as usual, $\lfloor x\rfloor$ is the integer part of the real $x$):

\medskip
{\it Case 1.} $\I0\neq\Iq$: let $0<s<q<1$, and take the set $A=\{a_1<a_2<\cdots\}\subset\NN$, where for all $n\in\NN$,
\[
a_n=\lfloor n^\frac{1}{s}\rfloor.
\]
Then $a_n=n^\frac{1}{s}-\varepsilon(n)$ for some $0\leq\varepsilon(n)<1$, and by Lagrange's Mean Value Theorem for  $f(x)=x^\frac{1}{s}$  on $[n,n+1]$ we get that $a_n<a_{n+1}$ for all $n$. Since
\[
\frac{\log n}{\log a_n} =\frac{\log n}{\frac{1}{s}.\log n+\log \Big(1-\frac{\varepsilon(n)}{n^\frac{1}{s}}\Big)}\to s,\quad\text{if} \ n\to\infty ,
\]
then $0<\lambda(A)=s<q$; thus,  $A\in\Iq\rsetminus\I0$.
It is also clear that $\Iq\rsetminus\I0$ is infinite, since for any $k\in\NN$ the sets
$A_k=\{ka_n: \ n\in\NN\}$ satisfy
\[
\lambda(A_k)=\limsup_{n\to\infty}\frac{\log n}{\log ka_n}=\lambda(A).
\]

{\it Case 2.} $\Iq\neq\Icq$: let $0<q<1$, and take the set $A=\{a_1<a_2<\cdots\}\subset\NN$, where for all $n\in\NN$,
\[
a_n=\lfloor n^\frac{1}{q} \log^\frac{2}{q} (n+1)\rfloor +1.
\]
One can easily show that $(a_n)$ is increasing sequence, and,

\[
 \sum_{n=1}^{\infty} \frac{1}{a_n^q} < \sum_{n=2}^{\infty} \frac{1}{n\log^2 n}  < \infty ,\ \text{thus, } A\in\Icq .
\]
On the other hand
\[
\lim_{n\to\infty}\frac{\log n}{\log a_n} =\lim_{n\to\infty} \frac{\log n}{\log (n^\frac{1}{q}\log^\frac{2}{q} (n+1))}  =\lim_{n\to\infty} \frac{\log n}{\frac1q\log n +\frac{2}{q}\log\log (n+1)} = q,
\]
hence, $\lambda(A)=q$.
Similarly to Case 1 we can see that $\Icq\rsetminus\Iq$ is actually infinite.

{\it Case 3.} $\Icq\neq \Ieq$: let $0<q<1$, define $A=\{a_1<a_2<\cdots\}\subset\NN$, where $a_n=\lfloor n^\frac{1}{q}\rfloor$ for all $n\in\NN$. Then
\[
\sum_{n=1}^{\infty} \frac{1}{a_n^q}\geq\sum_{n=1}^{\infty} \frac{1}{n}=\infty,
\]
so $A\notin\Icq$, but $A\in\Ieq$, since $\lambda(A)=q$.  Analogously to Case 1, one can show that $\Ieq\rsetminus\Icq$ is infinite.

{\it Case 4.} $\Ieq\neq\mathcal I_{<q'}$: it suffices to choose the set $A=\{a_1<a_2<\cdots\}\subset\NN$ such that $a_n=\lfloor n^\frac{1}{s}\rfloor$ for all $n$, where $0<q<s<q'$.
Then $\lambda(A)=s$, so $A\in\mathcal I_{<q'}$, however, $A\notin\Ieq$. Moreover, again,  $\mathcal I_{<q'}\rsetminus\Ieq$ is infinite. \qed
\end{proof}
\medskip

It is worth noting by \eqref{sorba}, that in order to decide if a given $A\subset \NN$ belongs to  $\Icq$, it may be easier, or more advantageous to first determine the convergence exponent of $A$.
Indeed, if $\lambda(A)<q$, then $A\in\Iq\subset\Icq$, or, if $\lambda(A)=q$ then $A\in\Ieq\subset\mathcal I_c^{(q')}$ for every $q'>q$. This view is important, since in what follows, we will establish criteria for  $\Iq$, $\Ieq$ membership, respectively.

\begin{Theorem} Let $0<q\leq 1$. Then each of the sets $\I0$, $\Iq$, $\Ieq$ forms an admissible ideal, except for $\mathcal I_{\leq 1}$.
\end{Theorem}

\begin{proof} Follows from properties of $\lambda$ listed in the Introduction, along with \eqref{sorba}. \qed
\end{proof}

\begin{Theorem}\label{metszet}
We have
\[
\I0=\bigcap_{0<q\leq 1} \Iq=\bigcap_{0<q\leq 1}\Ieq,
\]
hence,
\[
\I0=\bigcap_{0<q\leq 1}\Icq .
\]
\end{Theorem}
\begin{proof} Follows from the definitions of $\I0$, $\Iq$, $\Ieq$,  and \eqref{sorba}. \qed
\end{proof}

\section{ Conditions for a set $A$ to belong to $\Iq$, $\Ieq$}

Given $x\ge 1$, define the counting function of $A\subset \NN$ as
$$A(x)=\#\{a\leq x: a\in A \}.$$
 We have

\begin{Theorem}\label{tAxqd}
	Let $0\leq q<1$ be a real number, and $A\subset\NN$. Then
$A\in\Ieq$ if and only if for every $\delta>0$
\begin{equation}
\label{Axqd}
\lim_{x\to\infty}\frac{A(x)}{x^{q+\delta}}= 0 .
\end{equation}
 \end{Theorem}

\begin{proof}
Let $A=\{a_1<a_2<\dots \}$, and $A\in\Ieq$. Then
\[
\lambda(A) =\limsup_{n\to\infty}\frac{\log n}{\log a_n}\leq q,
\]
so for any $\delta > 0$ there is an $n_0\in\NN$ so that for all $n\geq n_0$
\[
\frac{\log n}{\log a_n}\leq q+\frac{\delta}{2}, \ \text{thus}  \ A(a_n)=n\leq a_n ^{q+\frac{\delta}{2}}.
\]
If $x$ is sufficiently large,  we can find $n\geq n_0$ with $a_n\leq x<a_{n+1}$, hence, $A(x)=n\leq x^{q+\frac{\delta}{2}}$. Consequently,
\[
0\leq \frac{A(x)}{x^{q+\delta}} \leq \frac{x^{q+\frac{\delta}{2}}}{x^{q+\delta}}=\frac{1}{x^{\frac{\delta}{2}}}\to 0, \quad\text{as} \ x\to\infty,
\]
which implies \eqref{Axqd} for every $\delta >0$.

Conversely, let $\delta >0$, and \eqref{Axqd} be true for some  $A=\{a_1<a_2<\dots \}$. Then
\[
\frac{A(a_n)}{a_n ^{q+\delta}}\to 0, \quad \text{as} \ n\to\infty,
\]
so there is  $n_1\in\NN$ such that for all $n\geq n_1$,  $n\leq a_n ^{q+\delta}$, thus,
\[
\frac{\log n}{\log a_n}\leq \frac{(q+\delta)\log a_n}{\log a_n} = q+\delta.
\]
Then for all $\delta >0$, $\lambda(A)\leq q+\delta$, hence, letting $\delta\to 0$, we get $\lambda(A)\leq q$, so, $A\in\Ieq$. \qed
\end{proof}

The definition of $\Ieq$\nobreakdash-convergence immediately yields
\begin{Corollary}
	Let $0\leq q<1$, $\varepsilon >0$, $L$ and $x_n$ be  real numbers for all $n\in\NN$, and  $ A_\varepsilon =\{n:|x_n - L|\geq\varepsilon \}$. Then $\Ilim{\Ieq} x_n=L$ if and only if for every $\varepsilon >0$ and $\delta >0$
$$	\lim_{x\to\infty}\frac{A_\varepsilon (x)}{x^{q+\delta}}= 0 .$$
\end{Corollary}

\begin{Theorem}\label{tAxkisebbq-d}
	Let $0<q\leq 1$ be a real number, and $A\subset\NN$. Then
	$A\in\Iq$ if and only if there exists $\delta>0$ such that
	\begin{equation}
	\label{Axq-d}
	\lim_{x\to\infty}\frac{A(x)}{x^{q-\delta}}= 0 .
	\end{equation}
\end{Theorem}

\begin{proof} Let   $A\in\Iq$. Then
\[
\lambda(A) =\limsup_{n\to\infty}\frac{\log n}{\log a_n}< q, \quad\text{where} \ A=\{a_1<a_2<\dots\}.
\]
For each $\delta >0$ with $0<\delta<\frac{1}{2}(q-\lambda(A))$ there is  $n_0\in\NN$ so that for all $n\geq n_0$,
\[
\frac{\log n}{\log a_n}\leq q-2\delta, \quad \text{thus,} \ n\leq a_n ^{q-2\delta},
\]
hence, for all $n\geq n_0$,
\[
A(a_n)=n\leq a_n ^{q-2\delta}.
\]
If $x$ is large enough, there exists some $n\geq n_0$ with $a_n\leq x<a_{n+1}$, so $A(x)=n\leq x^{q-2\delta}$. This implies
\[
0\leq \frac{A(x)}{x^{q-\delta}} \leq \frac{x^{q-2\delta}}{x^{q-\delta}}=\frac{1}{x^{\delta}}\to 0, \quad\text{as} \ x\to\infty,
\]
and \eqref{Axq-d} follows.

Conversely, let $\delta >0$ be such that \eqref{Axq-d} is true. Then by Theorem \ref{beagyazas}  and Theorem \ref{tAxqd}

$$A\in\mathcal I_{\leq q-\delta}\subset I_{< q}.$$
\qed

\end{proof}

The definition of $\Iq$\nobreakdash-convergence immediately yields
\begin{Corollary}
Let $0<q\leq 1$, $\varepsilon >0$, $L$ and $x_n$ be real numbers for all $n\in\NN$, and  $ A_\varepsilon =\{n:|x_n - L|\geq\varepsilon \}$. Then
$\Ilim{\Iq} x_n=L$ if and only if for every $\varepsilon >0$ there exists $\delta >0$ such that
		$$	\lim_{x\to\infty}\frac{A_\varepsilon (x)}{x^{q-\delta}}= 0 .$$
\end{Corollary}

As an application of the above results, we will show that an important number-theoretic set belongs to the smallest element of  \eqref{sorba}, namely $\I0$:

\begin{Lemma}
	Given $k\in\NN$, and arbitrary primes $p_1<p_2<\cdots <p_k$, denote
	\[
	D(p_1,p_2\dots,p_k)=\{n\in\mathbb N: n=p_1^{\alpha_1}p_2^{\alpha_2}\cdots p_k^{\alpha_k}, \alpha_i\geq 0, i=1,2,\dots, k\}\,.
	\]
Then
	\[
	D(p_1,p_2\dots,p_k)\in \I0\,.
	\]
\end{Lemma}
\begin{proof}
	For a number $x\geq2$ denote
	\[
	D(p_1,p_2\dots,p_k)(x)=\#\{n\leq x:n\in D(p_1,p_2\dots,p_k)\} .
	\]
	Then by [11, p.37, Exercise 15] we have
	\[
D(p_1,p_2\dots,p_k)(x)\leq\prod_{i=1}^{k} \Big (\frac{\log x}{\log p_i}+1 \Big )\leq \Big(\frac{2}{\log 2}\log x\Big)^k\,.
\]
From this, by Theorem \ref{tAxqd} for $q=0$ we get
\[
D(p_1,p_2\dots,p_k)\in \I0.  
\eqno\qed 
\]
\end{proof}

\section{On $\Iq$\nobreakdash- and $\Ieq$\nobreakdash-convergence of arithmetic functions}

First we recall some arithmetic functions, which we will investigate with respect to $\Iq$\nobreakdash-  and $\Ieq$\nobreakdash-convergence. We refer to the papers  \cite{2}, \cite{6}, \cite{10}, \cite{12}, \cite{14}, \cite{16}, \cite{17}, \cite{18} for definitions and properties of these functions.

Let $n=p_{1}^{\alpha_{1}}\cdot p_{2}^{\alpha_{2}}\cdots p_{k}^{\alpha_{k}}$ be the canonical representation of  $n\in \NN$. Define

\begin{itemize}
	\renewcommand\labelitemi{$\bullet$}
	\item $\omega(n)$ - the number of distinct prime factors of $n$ (i.e. $\omega(n)=k$),
	\item $\Omega(n)$ - the number of prime factors of $n$ counted with multiplicities (i.e. $\Omega(n)=\alpha_{1}+\dotsb+\alpha_{k}$),
	\item for $n>1$,
	\begin{equation*}
	h(n)=\min_{1\leq j \leq k}\alpha_{j},\quad H(n)=\max_{1\leq j \leq k}\alpha_{j}
	\end{equation*}
	and $h(1)=1$, $H(1)=1$, 
	\item $f(n)=\prod_{d\mid n} d$, and $f^{*}(n)=\frac{1}{n}f(n)$,
	\item $a_{p}(n)$ as follows: $a_{p}(1)=0$, and if $n>0$, then $a_{p}(n)$ is the unique integer $j\geq 0$ satisfying $p^{j}\mid n$, but $p^{j+1}\nmid n$ i.~e., $p^{a_{p}(n)}\Vert n$.
	\item $\gamma (n)$ - the number of all representations of a natural number $n$ in the form $n=a^b$, where $a, b$ are positive integers (see \cite{10}). Let
	\[
n=a_1^{b_1}=a_2^{b_2}=\dots=a_{\gamma (n)}^{b_{\gamma (n)}}
\]
	be all such representations of a given $n$, where $a_i, b_i\in\NN$.
	\item for $n>1$,
	\[
\tau(n)= b_1+b_2+\dots +b_{\gamma (n)},
\]
	\item $N(n)$ - the number of times the positive integer $n$ occurs in Pascal's triangle (see \cite{1} and \cite{17}).
\end{itemize}

Recall that $\Icq$\nobreakdash-convergence of the following sequences has been established in \cite{2}, \cite{3}, \cite{5}:

\begin{itemize}
\item[I.]  For $0<q\leq 1$ we have $\Ilim{\Icq}  \frac{h(n)}{\log n} = 0$  (see \cite[Th.8]{2}),
\item[II.] Only for $q=1$ we have $\Ilim{\Icq} \frac{H(n)}{\log n} = 0$ (see \cite[Th.10, Th.11]{2}),
\item[III.] For a prime number $p$ the sequence $\big((\log p) \frac{a_p (n)}{\log n}\big)_{n=2}^\infty$ is $\Icq$\nobreakdash-convergent to $0$ only for $q=1$ (see \cite[Th.2.3]{3}),
\item[IV.] For $q>\frac{1}{2}$ we have $\Ilim{\Icq}  \gamma (n) = 1$, and  for $0<q\leq\frac{1}{2}$ the sequence $\gamma (n)$ is not $\Icq$\nobreakdash-convergent (see \cite[Cor.3.5]{3}),
\item[V.] For $q>\frac{1}{2}$ we have $\Ilim{\Icq}  \tau(n) = 1$, and for $0<q\leq\frac{1}{2}$ the sequence $\tau(n)$ is not $\Icq$\nobreakdash-convergent (see \cite[Cor.3.8]{3}),
\item[VI.] For $q>\frac{1}{2}$ we have $\Ilim{\Icq}  N(n) = 2$, and  for $0<q\leq\frac{1}{2}$ the sequence $\big(N(n)\big)_{t=1}^{\infty}$ is not $\Icq$\nobreakdash-convergent (see \cite[Th.2.2]{5}),
\item[VII.] The sequences $\big(\frac{\omega(n)}{\log \log n}\big)_{n=2}^\infty$ and $\big(\frac{\Omega(n)}{\log \log n}\big)_{n=2}^\infty$ are not $\Icq$\nobreakdash-convergent for all $0<q\leq 1$ (see \cite[Th.12]{2}),
\item[VIII.] The sequences $\big(\frac{\log\log f(n)}{\log\log n}\big)$ and $\big(\frac{\log\log f^{*}(n)}{\log\log n}\big)$ are not $\Icq$\nobreakdash-convergent for all $0<q\leq 1$ (see \cite[Th.13, Th.14]{2}).
\end{itemize}

In what follows, we will improve and sharpen all the statements I--VIII via the best convergences  one can obtain from the ideals in  \eqref{sorba}  that are within $\Iq$, $\Ieq$.

The next theorem, which is readily implied by Theorem \ref{metszet} and \cite[Th.8]{2},   gives statement I. using Theorem \ref{beagyazas} and Lemma 1.
 We will, however, provide another simpler proof based on Lemma 2:

\begin{Theorem}
	We have
	\[
	\Ilim{\I0} \frac{h(n)}{\log n}=0\,.
	\]
\end{Theorem}

\begin{proof} Take a small $\varepsilon>0$, and the largest prime  $p_0$ for which $\frac1{\log p_0}\geq \varepsilon$. Then $\frac1{\log p}< \varepsilon$ whenever $p>p_0$, so if $n\in\NN$ is such that $p|n$ for some prime  $p>p_0$, then $n\geq p^{h(n)}$. It follows that
	\[
	\frac{h(n)}{\log n}\leq \frac{h(n)}{\log p^{h(n)}}=\frac1{\log p}<\varepsilon\,,
	\]
	thus,
\[n\notin \Big\{k\in\mathbb N:\frac{h(k)}{\log k}\geq \varepsilon\Big\}=\Big\{k\in\mathbb N:\Big|\frac{h(k)}{\log k}-0\Big|\geq \varepsilon\Big\}=A_\varepsilon.
\]
This implies $A_\varepsilon\subset D(2,3,5,\dots, p_0)$, so, by  Lemma 2 and the hereditary property, $A_\varepsilon\in\I0$. \qed
\end{proof}

Statement II. has the following strengthening:

\begin{Theorem}
	We have
	\[
	\Ilim{\mathcal I_{<1}} \frac{H(n)}{\log n}=0\,.
	\]
\end{Theorem}

\begin{proof}
	
	Let $0<\varepsilon <\frac{1}{\log 2}$. Then, according to \eqref{Ae}, we have 
	$$A_\varepsilon=\big\{n\in \NN: \frac{H(n)}{\log n}\geq\varepsilon\big\}.$$ We will show that $A_\varepsilon\in \mathcal I_{<1}$: every positive integer $n$ can be uniquely represented as $n=ab^{2}$, where $a$ is a
	square-free number. Hence $H(a)=1$ and $H(n)\in \{H(b^2), H(b^2)+1\}$. For any $n\in\NN$ we have $n=p_1^{a_1}\cdots p_k^{a_k} \geq 2^{H(n)}$ and from this
\[
H(n)\leq \frac{\log n}{\log 2}.
\]
If $n\in A_\varepsilon$ then for $n=ab^2$ we get
\[
\log n=\log(ab^2)\leq \frac{H(ab^2)}{\varepsilon}\leq \frac{H(b^2)+1}{\varepsilon}\leq \frac{\log b^2}{\varepsilon\log 2}+\frac{1}{\varepsilon},
\]
thus,
\[
A_\varepsilon\subseteq B=\Big\{n\in\NN :  n=ab^2,\  \log ab^2 \leq \frac{\log b^2}{\varepsilon\log 2}+\frac{1}{\varepsilon},\ \  a,b\in\NN\Big\}.
\]
	Furthermore, if $n\in B$, then
\[
\log a\leq \frac{1-\varepsilon\log 2}{\varepsilon\log 2}\log b^2 +\frac{1}{\varepsilon},
\]
	which is equivalent to 
\[\mathlarger{
a^\frac{\varepsilon\log 2}{1-\varepsilon\log 2}\leq b^2 e^\frac{\log 2}{1-\varepsilon\log 2} },\quad\textrm{and so }	
\mathlarger{	\quad a^\frac{1}{1-\varepsilon\log 2}\leq a b^2 e^{\frac{\log 2}{1-\varepsilon\log 2}}   },
\]
therefore,
\[
B=\left\{n\in\NN :  n=ab^2  \textrm{ and }  a\leq 2n^{1-\varepsilon\log 2} \right\}.
\]
	If $n\in B$, and $n=ab^2\leq x$ for $x\geq 2$, then $a\le 2x^{1-\varepsilon\log 2} $ and $b\leq\sqrt{\frac{x}{a}}$.
Consequently,
\begin{gather*}
B(x)\leq \sum_{a<2x^{1-\varepsilon\log 2}} \sqrt{\frac{x}{a}}=\sqrt{x}\sum_{a<2x^{1-\varepsilon\log 2}}\frac{1}{\sqrt{a}}\leq \sqrt{x} \bigg(1+\int_{1}^{ 2x^{1-\varepsilon\log 2}} \frac{1}{\sqrt{t}} \mathrm dt \bigg)
\\
	\leq\sqrt{x}\Big(1+2\big(\sqrt{2x^{1-\varepsilon\log 2}}-1\big)\Big)\leq 2\sqrt{2} x^{1-\varepsilon\frac{\log 2}{2}},
\end{gather*}
	hence, for $x\geq 2$, we have
\[
A_\varepsilon (x)\leq 2\sqrt{2}  x^{1-\varepsilon\frac{\log 2}{2}}.
\]
Using  $q= 1$ and arbitrary $\delta\in(0,\varepsilon\frac{\log 2}{2})$ in Theorem \ref{tAxkisebbq-d}, the above estimate gives $A_\varepsilon\in \mathcal I_{<1}$.
 \qed
\end{proof}

The next result strengthens statement III:

\begin{Theorem}
	For any prime number $p$, we have
	\[
	\Ilim{\mathcal I_{<1}} (\log p)\frac{a_p (n)}{\log n}=0\,.
	\]
\end{Theorem}

\begin{proof}
Let $0<\varepsilon <1$. Then, according to \eqref{Ae}, we have  $$A_\varepsilon =\{n>1: (\log p)\frac{a_p (n)}{\log n}\geq \varepsilon \}.$$ We have
\begin{gather*}
A_\varepsilon =\bigcup_{i=0}^{\infty} A_\varepsilon^{i},\quad\text{where}\\
A_\varepsilon^{i} = \{n\in A_\varepsilon :  n=p^i u \quad\text{where} \  p\nmid u \} \  (i=0, 1,2 \dots).
\end{gather*}
Clearly, $A_\varepsilon^{i}\cap A_\varepsilon^{j} =\emptyset $ for $i\neq j$, and if $n\in A_\varepsilon^{i}$, then
\[
(\log p)\frac{a_p (n)}{\log n}=(\log p)\frac{i}{i\log p +\log u}\geq \varepsilon, \quad\text{thus, } u\leq p^{i(\frac{1-\varepsilon}{\varepsilon})}.
\]
This implies, in case $x\geq 2$, that
\[
A_\varepsilon^{i} (x)\leq \# \left\{u:u^{\frac{\varepsilon}{1-\varepsilon}}u\leq x \right\}=\# \left\{u:u^{\frac{1}{1-\varepsilon}}\leq x \right\}\leq x^{1-\varepsilon},
\]
hence,
\[
A_\varepsilon (x)=\sum_{i: p^i\leq x} A_\varepsilon ^{i} (x)\leq \frac{\log x}{\log p } x^{1-\varepsilon}.
\]
Using  $q= 1-\varepsilon$ in Theorem \ref{tAxqd} and using Theorem \ref{beagyazas}, the above estimate gives
\[
A_\varepsilon\in \mathcal I_{\le 1-\varepsilon} \subset \mathcal I_{<1}.
\]
\qed
\end{proof}

The statements IV., V., VI. are consequences of the following:
\begin{Theorem}
	We have
\begin{itemize}
\item[i)]  $\Ilim{\mathcal I_{\leq \frac12}}  \gamma (n) = 1.$	
\item[ii)]  $\Ilim{\mathcal I_{\leq \frac12}} \tau(n) = 1.$
\item[iii)]  $\Ilim{\mathcal I_{\leq \frac12}}  N(n) = 2.$
\end{itemize}
\end{Theorem}

\begin{proof}i) Let $0<\varepsilon <1$. Then, according to \eqref{Ae}, we have   $A_\varepsilon =\{n\in\NN: |\gamma(n)-1|\geq \varepsilon \}$. Clearly,
\[
A_\varepsilon \subseteq H=\big\{a^b: a, b\in\NN\rsetminus\{1\}\big\}=\bigcup_{k=2}^\infty \big\{n^k:n=2,3,\dots\big\}.
\]
Given some $x\in\NN$, $x\geq 2^2$, there is a $k\in\NN\rsetminus\{1\}$ with  $2^k\leq x <2^{k+1}$. Then $k\leq \frac{\log x}{\log 2}$, and
\[
H(x)\leq \sum_{n=2}^{k} \sqrt[n]{x}\leq \sqrt{x}  \frac{\log x}{\log 2},
\]
thus, for all $x\geq 4$,
\[
A_\varepsilon (x)\leq \frac{\log x}{\log 2} x^{\frac{1}{2}}.
\]
For $q=\frac{1}{2}$ in Theorem \ref{tAxqd}, we get $A_\varepsilon\in \mathcal I_{\leq \frac12 }$.

ii) Similar to i).

iii) Let $0<\varepsilon <1$. Then, according to \eqref{Ae}, we have  $A_\varepsilon =\{n\in\NN: |N(n)-2|\geq \varepsilon \}$. If we take $H=\{1,2\}\cup M$, where $M=\{n\in\NN : N(n)> 2\}$, then $A_\varepsilon\subset H$. It has been proved in \cite{1}, that $M(x)=O(\sqrt{x})$, thus, there is a $c>0$ so that for all $x\geq 2$,
\[
A_\varepsilon (x)\leq H(x)\leq cx^{\frac{1}{2}}.
\]
It now follows, by Theorem \ref{tAxqd}, that $A_\varepsilon\in \mathcal I_{\leq \frac12 }$. \qed
\end{proof}

\begin{Remark} We note, that the set  $\mathcal I_d$ containing all subsets of $\NN$ with zero asymptotic density forms an admissible ideal. The corresponding $\mathcal I_d$\nobreakdash-convergence is the wellknown statistical convergence. The following results were proved in \cite{16} and \cite{15}:
\[
\Ilim{\mathcal I_d} \frac{\omega(n)}{\log \log n}
=\Ilim{\mathcal I_d} \frac{\Omega(n)}{\log \log n}
=1, \]
\[
\Ilim{\mathcal I_d} \frac{\log\log f(n)}{\log \log n}
=\Ilim{\mathcal I_d} \frac{\log\log f^*(n)}{\log \log n}
=1+\log 2.
\]	
We note, that $\mathcal I_c^{(1)}\subsetneq \mathcal I_d$.

	 If $\Ilim{\Icq}  x_n =L$ is false for every $0<q\leq 1$, then $(x_n)$ does not $\Iq$\nobreakdash-converge for any $q$, so  $A_\varepsilon =\{n\in\NN : |x_n - L|\geq \varepsilon\}\notin \Iq$ whenever  $0<q\leq 1$; thus, $\lambda(A_\varepsilon) = 1$ is the only option. Then by VII. and  VIII. it follows that for all $\varepsilon >0$ and for every $n $,  $a_n\in\{\omega (n), \Omega(n)\}$, and  $b_n\in\{f(n), f^*(n)\}$ we have
\begin{itemize}
	\item[i)]  $\lambda\Big(\big\{n\in\NN : \big|\frac{a_n}{\log\log n}-1\big|\geq \varepsilon\big\}\Big)=1 ,$	
	\item[ii)]  $\lambda\Big(\big\{n\in\NN : \big|\frac{\log\log b_n}{\log\log n}-(1+\log 2)\big|\geq \varepsilon\big\}\Big)=1.$
	\end{itemize}
As a consequence, say of i) for $a_n =\omega (n)$, we have that if
\[
\Big\{n\in\NN : \Big|\frac{\omega(n)}{\log\log n}-1\Big|\geq \varepsilon\Big\} = \{n_1<n_2 <\dots <n_k<\dots\},
\]
then
\[
\limsup_{k\to\infty} \frac{\log k}{\log n_k} =1.
\]
\end{Remark}

\end{document}